\documentclass[12pt]{article}
\usepackage{MyStyleSheet}
\usepackage{hyperref}
\usepackage{tikz}
\usepackage{amsfonts}
\usepackage{braids}
\usepackage{blindtext}
\usepackage{geometry}
\usepackage{tikz}
\usepackage{caption}
\usepackage{theoremref}

\makeatletter
\let\@fnsymbol\@arabic
\makeatother
\title{Decidability of the Orbit Problem Over $\mathbb{Q}(X)$}
\author{Joshua Holden\thanks{Joshua Holden: holden@rose-hulman.edu; Department of Mathematics, Rose-Hulman Institute of Technology, 5500 Wabash Ave., Terre Haute, IN 47803, USA.}, Alexa Renner\footnote{Corresponding Author: Alexa Renner: renneram@rose-hulman.edu; Department of Mathematics, Rose-Hulman Institute of Technology, 5500 Wabash Ave., Terre Haute, IN 47803, USA.}}
\date{August 2026}

\newtheorem{counter}{counter}[section]

\theoremstyle{definition}
\newtheorem{definition}[counter]{Definition}
\newtheorem*{definition*}{Definition}

\theoremstyle{plain}
\newtheorem{theorem}[counter]{Theorem}
\newtheorem{lemma}[counter]{Lemma}
\newtheorem{corollary}[counter]{Corollary}

\theoremstyle{remark}

\begin{document}
\maketitle
\newcommand{\generation}{\text{generation}}
\newcommand{\Triv}{\text{Triv}}
\newcommand{\Frac}{\text{Frac}}
\newcommand{\CellWidth}{\text{CellWidth}}
\newcommand{\height}{\text{height}}
\newcommand{\pre}{\text{pre}}
\newcommand{\per}{\text{per}}
\newcommand{\post}{\text{post}}
\newcommand{\start}{\text{start}}
\newcommand{\send}{\text{end}}
\newcommand{\genTurn}{\text{gen}^\dag}
\newcommand{\ind}{\text{ind}}
\newcommand{\length}{\text{length}}
\newcommand{\ts}{\text{ts}}
\newcommand{\ct}{\text{ct}}
\newcommand{\width}{\text{width}}
\newcommand{\word}{\text{word}}
\newcommand{\cs}{\text{cs}}
\newcommand{\cross}{\text{cross}}
\newcommand{\In}{\text{In}}
\begin{abstract}
    The orbit problem is the problem of whether, given $x, y\in \mathbb{Q}^n$ and an $n\times n$ matrix $A$, there exists an $i\in\mathbb{N}$ such that $A^ix = y$. In 1980, Kannan and Lipton proved that the orbit problem is decidable. We show that a generalization of the orbit problem, where the field is $\mathbb{Q}(X)$ for $X$ a countable set of transcendentals, is also decidable. We define the orbit power problem for an arbitrary group $G$ to be the problem of when, given $x, y\in G$, there exists an $n\in\mathbb{Z}$ such that $x^n = y$. We then use the main result to show that the power orbit problem is decidable for an assortment of groups, including the braid groups and $\operatorname{Aut}(F_2)$.
\end{abstract}

\hfill\break
\section*{Introduction}
The orbit problem, or the problem of whether, given $x, y\in \mathbb{Q}^n$ and an $n\times n$ matrix $A$, there exists an $i\in\mathbb{N}$ such that $A^ix = y$, was shown to be decidable in \cite{KL}. Since then, several generalizations of the orbit problem have emerged and have been solved. For example, in \cite{a}, the problem was generalized to the following:
\begin{center}
    Given $A$ a $d\times d$ matrix and semialgebraic sets $S, T\subseteq\mathbb{R}^d$, does there exist an $x\in S$ and $n\in\mathbb{N}$ such that $A^nx \in T$?
\end{center}
and was shown to be decidable for $d\leq 3$. The orbit problem has even been generalized to arbitrary groups, see \cite{enric}. We consider the generalization of the orbit problem to $\mathbb{Q}(X)$ for $X$ a countable set of indeterminants.

In section \ref{orbitproblem}, we prove that the orbit problem over $\mathbb{Q}(X)$ for $X$ a countable set of indeterminants is decidable, borrowing heavily from \cite{KL}. In section \ref{powerorbitproblem}, we apply this result to prove that in certain groups that admit certain linear representations, a generalization of the orbit problem is decidable. Finally, section \ref{outlook} gives a few different directions for future work.
\subsection*{Notation and Conventions}
We assume $0\in\mathbb{N}$. Almost all of the notation listed here is compatible with the notation in \cite{KL}.
\begin{itemize}
    \item Given fields $L/K$ and $\alpha\in L$ algebraic over $K$, we define the minimal polynomial of $\alpha$ to be the monic polynomial of minimal degree that has $\alpha$ as a root, and denote it by $f_\alpha$.
    \item We denote the degree of the minimal polynomial of $\alpha$ by $n_\alpha$. The fields in question will be made clear from context.
    \item Let $X$ be a set. We denote the set of functions from $X$ to $X$ by ${^X}\!X$.
    \item Let $n\in\mathbb{N}^+$. We denote the set $\{1,\cdots,n\}$ by $[n]$.
    \item We denote the set of $n\times n$ matrices over a ring $R$ by $M_{n\times n}(R)$.
    \item Given a matrix $A$, we denote its minimal polynomial by $f_A$.
    \item Let $X$ be a countable set of indeterminants and let $f(t_{\beta_1},\cdots,t_{\beta_m})(x)\in \mathbb{Q}(X)[x]$. We will write $f(x)$ in place of $f(t_{\beta_1},\cdots,t_{\beta_m})(x)$, and we will write $f|_{(d_1,\cdots,d_m)}(x)$ in place of $f(d_1,\cdots,d_m)(x)$ for $(d_1,\cdots,d_m)\in\mathbb{Q}^m$.
    \item Let $q(x)\in\mathbb{Q}[x]$, and let $q(x) = \frac{a_n}{b_n}x^n + \cdots + \frac{a_0}{b_0}$ where the coefficients are in lowest terms. We denote $\prod_{0\leq i\leq n} \left(|a_n||b_n|\right)$ by $||q||$\footnote{This definition differs slightly from the corresponding definition given in \cite{KL}, but this is necessary for the proof of \thref{21analog}.}.
\end{itemize}
\section{Orbit Problem}\label{orbitproblem}
Let $X$ be a countable set of indeterminants. By replacing $\mathbb{Q}$ with $\mathbb{Q}(X)$ in \cite[Section 1]{KL}, we get that the problem:
\begin{center}
    Given $A\in M_{n\times n}(\mathbb{Q}(X))$ and $x,y\in\mathbb{Q}(X)^n$, does there exist an $i\in\mathbb{N}$ such that $A^ix = y$?
\end{center}
is reducible to the problem:
\begin{center}
    Given $A,B\in M_{n\times n}(\mathbb{Q}(X))$, does there exist an $i\in\mathbb{N}$ such that $A^i = B$?
\end{center}
which is in turn reducible to the problem:
\begin{center}
    Given $A\in M_{n\times n}(\mathbb{Q}(X))$ and $q(x)\in \mathbb{Q}(X)[x]$, does there exist an $i\in\mathbb{N}$ such that $A^i = q(A)$?
\end{center}
which we will show is decidable by generalizing some of the arguments in \cite{KL}.

We first need to cite a folklore lemma:
\begin{lemma}\thlabel{rootsofunitylemma}
    Let $g(x)\in\mathbb{Z}[x]$ be such that $g(x) = x^n + a_{n-1}x^{n-1}+\cdots + a_0$. If all roots of $g(x)$ are roots of unity, then $|a_i|\leq\binom{n}{i}$ for $0\leq i\leq n-1$.
\end{lemma}
The following generalizes \cite[Theorem 2.1]{KL} from $\mathbb{Q}$ to $\mathbb{Q}(t_1,\cdots, t_m)$, for $t_1,\cdots,t_m$ indeterminants:
\begin{lemma}\thlabel{21analog} 
    Fix $m\in\mathbb{N}$ and let $t_1,\cdots, t_m$ be indeterminants. There exists an exactly computable function $p(\cdot,\cdot)$ such that for any number $\alpha\in\overline{\mathbb{Q}(t_1,\cdots, t_m)}$ and $q(x)\in \mathbb{Q}(t_1,\cdots,t_m)[x]$, if $\alpha$ is not a root of unity, then $\alpha^i = q(\alpha)$ implies $i\leq p(f_\alpha,q)$.
\end{lemma}
\begin{proof}
    Let $K = \operatorname{Frac}(\mathbb{Z}[t_1,\cdots,t_m]) = \mathbb{Q}(t_1,\cdots, t_m)$. We consider three cases for $\alpha$: The case where $\alpha$ is not an algebraic integer over $K$, the case where $\alpha$ is an algebraic integer over $K$ but is not an element of $\overline{\mathbb{Q}}$, and the case where $\alpha\in\overline{\mathbb{Q}}$. In each case, we prove that there exists a bound on the $i$ for which $\alpha^i = q(\alpha)$, and this bound can be computed using only $f_\alpha$ and $q$.

    Let $f_\alpha(x) = x^{n_\alpha} + \frac{r_{n_\alpha-1}(t_1,\cdots,t_m)}{s_{n_\alpha-1}(t_1,\cdots,t_m)}x^{n-1} + \cdots +\frac{r_0(t_1,\cdots,t_m)}{s_0(t_1,\cdots,t_m)}$ be the minimal polynomial of $\alpha$ for $r_i(t_1,\cdots,t_m),s_i(t_1,\cdots,t_m)\in\mathbb{Z}[t_1,\cdots,t_m]$ polynomials. We may assume that $\gcd(s_i(t_1,\cdots,t_m),r_i(t_1,\cdots,t_m)) = 1$ for $0\leq i\leq n-1$. We first prove a claim:

    \hfill\break
    \textbf{Claim 1:} Let $(d_1,\cdots,d_m)\in\mathbb{Z}^m$, $q(x)\in K[x]$, and let $\Tilde{\alpha}$ be a root of $q|_{(d_1,\cdots, d_m)}$. For any $k\in\mathbb{N}$, if $\alpha^k = q(\alpha),$ then $ \Tilde{\alpha}^k = q|_{(d_1,\cdots,d_m)}(\Tilde{\alpha})$.

    \hfill\break
    \textbf{Proof:} Suppose $\alpha^k = q(\alpha)$. Then $\alpha$ is a root of $x^k - q(x)$, so there exists $w(x)\in K[x]$ with $x^k - q(x) = f_\alpha(x)w(x)$. Then, $x^k - q|_{(d_1,\cdots,d_m)}(x) = f_\alpha|_{(d_1,\cdots,d_m)}(x)w|_{(d_1,\cdots,d_m)}(x)$. Because $f_\alpha|_{(d_1,\cdots,d_m)}(\Tilde{\alpha})=0,$ we have that $\Tilde{\alpha}^k-q|_{(d_1,\cdots,d_m)}(\Tilde{\alpha}) = 0$, so $\Tilde{\alpha}^k = q|_{(d_1,\cdots,d_m)}(\Tilde{\alpha})$.
\begin{flushright}
$\dashv_{\text{Claim}}$
\end{flushright}

    \hfill\break
    \textbf{Case 1:} $\alpha$ is not an algebraic integer over $K$.

    By definition of algebraic integer, the $s_i$ are not identically 1. Then, by the fundamental theorem of algebra, there exist $d_1,\cdots,d_m\in\mathbb{Z}$ such that for all $i, s_i(d_1,\cdots,d_m)\neq 0$ and there exists an $i$ such that $s_i(d_1,\cdots,d_m)\neq1$. Notice that we can find $(d_1,\cdots,d_m)$ in a finite time as $\mathbb{Z}^m$ is enumerable. Let $\Tilde{\alpha}$ be any root of $f_\alpha|_{(d_1,\cdots,d_m)}(x)$ which is not an algebraic integer over $\mathbb{Q}$, which exists by the fact that $f_\alpha|_{(d_1,\cdots,d_m)}$ is a product of the monic minimal polynomials of its roots.

    Let $p_1(f_\alpha, r) = n_\alpha2\log_2||f_\alpha|_{(d_1,\cdots,d_m)}||$\footnote{\cite{KL} defines $||\cdot||$ differently and cites \cite{marcus} for the result. However, this result fails for the norm given in \cite{KL}.}. By claim 1 and the proof of \cite[Theorem 2.1]{KL}, $i\leq  p_1(f_\alpha,r)$.

\hfill\break
\textbf{Case 2:} $\alpha$ is an algebraic integer over $K, \alpha\not\in\overline{\mathbb{Q}}$. 

Because there exists at least one $r_j$ that is nonconstant, and because $r_j(t_1,\cdots,t_m)\in\mathbb{Z}[t_1,\cdots,t_m]$, there exists a $(d_1,\cdots,d_m)\in\mathbb{Z}^m$ such that $r_j(d_1,\cdots,d_m)>\binom{n}{i}$. By \thref{rootsofunitylemma}, there exists a root $\Tilde{\alpha}$ of $f_\alpha|_{(d_1,\cdots,d_m)}(x)$ which is not a root of unity. Then, if $\alpha^i = q(\alpha)$, we must have $\Tilde{\alpha}^i = q|_{(d_1,\cdots,d_m)}(\Tilde{\alpha})$ by claim 1. By \cite{bbm}, there exists a conjugate $\theta$ of $\Tilde{\alpha}$ such that $|\theta|>1 + \frac{1}{(30n_\theta^2\log_2(6n_\theta))}\geq 1 + \frac{1}{(30n_\alpha^2\log_2(6n_\alpha))}$ as $n_\theta\leq n_\alpha$ because $f_\theta$ divides $f_\alpha|_{(d_1,\cdots,d_m)}$. Then, $i\leq \frac{\log|q|_{(d_1,\cdots,d_m)}(\theta)|}{\log|\theta|}$. Because $\theta$ is a conjugate of $\Tilde{\alpha},$ we have that $ \Tilde{\alpha}^i = q|_{(d_1,\cdots,d_m)}(\Tilde{\alpha})$ implies that $\theta^i = q|_{(d_1,\cdots,d_m)}(\theta)$. By the fact that $|\theta| > 1$:
\begin{gather*}
    \log|q(\theta)|\leq\log[(n_\theta+1)|\theta|^{n_\theta}||q|_{(d_1,\cdots,d_m)}||] = \log(n_\theta + 1) + n_\theta\log|\theta| +\log||q|_{(d_1,\cdots,d_m)}||\\
    \leq \log(n_\alpha + 1) + n_\alpha\log|\theta| +\log||q|_{(d_1,\cdots,d_m)}||.
\end{gather*}
Define $p_2(f_\alpha,q) = (30n_\alpha^2\ln(6n_\alpha))(\log(n_\alpha+1) + \log||q|_{(d_1,\cdots,d_m)}||)+ n_\alpha$. Then, 
\begin{align*}
    i\leq\frac{\log(n_\alpha+1)+\log||q|_{(d_1,\cdots,d_m)}||}{\log|\theta|}+n_\alpha\leq (30n_\alpha^2\ln(6n_\alpha))(\log(n_\alpha+1) + \log||q|_{(d_1,\cdots,d_m)}||)+ n_\alpha\\
    = p_2(f_\alpha, q).
\end{align*}

\hfill\break
\textbf{Case 3:} $\alpha$ is an algebraic integer over $K$ in $\overline{\mathbb{Q}}$.

By \cite[Theorem 2.1]{KL}, there exists a polynomial $p_3(f_\alpha,q)$ such that $i\leq p_3(f_\alpha,q)$.

Set $p(f_\alpha, q) = \max\{\lceil p_1(f_\alpha,q)\rceil,\lceil p_2(f_\alpha,q)\rceil,\lceil p_3(f_\alpha,q)\rceil\}$. Because $p_1,p_2,p_3$ are each computable, $p(f_\alpha, q)$ is computable. By cases 1-3, if $\alpha$ is not a root  of unity and $q(x)\in K[x], $ then $\alpha^i = q(\alpha)$ implies that $i\leq p(f_\alpha, q)$.    
\end{proof}
\begin{lemma}\thlabel{big21analog}
    Let $X$ be a countable set of indeterminants. Let $\alpha$ be algebraic over $\operatorname{Frac}(\mathbb{Z}[X])$. There exists a computable function $p(f_\alpha,q)$ such that if there exists $i$ with $\alpha^i = q(\alpha)$, then $i\leq p(f_\alpha,q)$ if $\alpha$ is not a root of unity.
\end{lemma}
\begin{proof}
    It is a computable process to determine the finite set of indeterminants $\{t_1,\cdots, t_m\}$ which appear in $f_\alpha(x)$ or $q(x)$. The lemma follows from \thref{rootsofunitylemma}.
\end{proof}
We now restate \cite[Theorem 3.1]{KL}:
\begin{lemma}\thlabel{31analog}
    Let $\mathbb{F}$ be any field and let $X$ be a countable set of indeterminants. Suppose $A\in\mathbb{F}^{n\times n}$ has minimal polynomial $p(x)$ belonging to $\mathbb{F}[x]$ and let $r(x),q(x)\in\mathbb{F}[x]$. Then:
    \begin{enumerate}
        \item $r(A) = q(A)$ if and only if $r(x) = q(x)\mod p(x)$.
    \end{enumerate}
    Furthermore, if $F = \mathbb{Q}(X)$ and $p(x)$ is irreducible over $\mathbb{Q}(X)$ and has $\alpha$ as a root, then (1) is equivalent to $r(\alpha) = q(\alpha)$.
\end{lemma}
Notice that the proof is identical to that given in \cite{KL}, with $\mathbb{Q}$ replaced with $\mathbb{Q}(X)$.
\begin{lemma}\thlabel{reducetocount}
   Let $X$ be a countable set of indeterminants and let $p(x)\in\mathbb{Q}(X)[x]$. Then, if $p(x)\in\mathbb{Q}(t_1,\cdots,t_n)[x]$ for some $t_1,\cdots,t_n\in X$, then $p(x)$ is reducible in $\mathbb{Q}(X)[x]$ if and only if it is reducible in $\mathbb{Q}(t_1,\cdots,t_n)[x]$.
\end{lemma}
\begin{proof}
    The proof follows from the definitions.
\end{proof}
Notice that the minimal polynomial algorithm and the roots of unity algorithms given in \cite{KL} can be used to find the minimum polynomial of a matrix over $K$ or to determine if all roots of a polynomial in $\mathbb{Q}(X)[x]$ are roots of unity, respectively. Using the roots of unity algorithm given in \cite[Section 3, Page 258]{KL} and \cite[Lemma 3.2]{KL}, we can determine if all roots of $f_A$ are roots of unity.

\hfill\break
\textbf{Case 1:} All roots of $f_A$ are roots of unity.

Notice that $f_A\in\mathbb{Q}[x]$ and $f_A$ is monic. From the proofs of cases 2 and 3 of \cite{KL}, \thref{31analog}, and \thref{reducetocount}, it follows that the algorithms described in cases 2 and 3 of \cite[Section 3]{KL} determine whether there exists an $i\in\mathbb{N}$ such that $A^i = q(A)$.

\hfill\break
\textbf{Case 2:} There is a root of $f_A$ that is not a root of unity.

Notice that $A^i=q(A)$ implies $x^i\equiv q(x)\mod f_A(x)$, which in turn implies that $x^i \equiv q(x)\mod f_\alpha(x)$ as $f_\alpha(x)\mid f_A(x)$. It is an exactly computable process to enumerate the indeterminants $t_1,\cdots, t_m$ that appear in $f_A$. By \thref{31analog}, $\alpha^i = q(\alpha)$. By \thref{21analog}, $i\leq p(f_\alpha,q)$. Notice that, by \thref{reducetocount}, the unique factorization of $f_A$ can be done over $\mathbb{Q}(t_1,\cdots,t_m)$. Polynomial division in $\mathbb{Q}(t_1,\cdots,t_m)[x]$ is exactly computable and the monic polynomials over $\mathbb{Q}(t_1,\cdots,t_m)$ of degree $\deg f_A$ or less are enumerable, so we can determine the minimal polynomials of roots of $f_A$. Enumerate the minimal polynomials as $f_{\alpha_1},\cdots,f_{\alpha_k}$. Then, let $r = \max\{ p(f_{\alpha_1},q),\cdots,p(f_{\alpha_k},q)\}$, and for each $0\leq i\leq r$, determine if $A^i = q(A)$. If such an $i$ exists, accept; otherwise, reject.

The above implies the following:
\begin{theorem}\thlabel{touseLk}
    Let $X$ be any countable set of indeterminants, and let $A,B\in M_{n\times n}(\mathbb{Q}(X))$. The problem of whether there exists $i\in\mathbb{N}$ such that $A^i = B$ is decidable.
\end{theorem}
\begin{corollary}\thlabel{indcor}
    Let $X$ be any countable set of indeterminants. The orbit problem for $M_n(\mathbb{Q}(X))$ is decidable.
\end{corollary}
\section{A Specific Instance of the Orbit Problem for Groups}\label{powerorbitproblem}
The orbit problem can be generalized to arbitrary sets in the following way:
\begin{definition}
    \cite[Definition 1.1]{enric}: Let $X$ be a set and let $A\subseteq {^X}\!X$. We say that $A$ is \textbf{orbit decidable} if there is an algorithm which, given $x, y\in X$, decides whether $\alpha(x) = y$ for some $\alpha\in A$. The search version of the problem asks the algorithm to provide such an $\alpha$.
\end{definition}
This definition is very broad. This, together with the reduction of the problem
\begin{center}
    Given $A\in GL_n(\mathbb{Q})$ and $x,y\in \mathbb{Q}^n$, does there exist an $i\in\mathbb{N}$ such that $A^ix = y$
\end{center}
to the problem:
\begin{center}
    Given $A, B\in GL_n(\mathbb{Q})$, does there exist an $i\in\mathbb{N}$ such that $A^i = B$?
\end{center}
 given in \cite[Section 1]{KL} motivates the following somewhat restricted version of the orbit problem:
\begin{definition}
    Let $G$ be a group. We call the question of whether, given $x, y\in G$, there exists a $n\in\mathbb{Z}$ such that $x^n = y$ the \textbf{power orbit problem} in $G$. We call the question of whether, given $x, y\in G$, there exists a $n\in\mathbb{N}$ such that $x^n = y$ the \textbf{nonnegative power orbit problem}.
\end{definition}
\begin{corollary}\thlabel{indcorrep}
    Let $X$ be a countable set of indeterminants. Suppose $G$ is a group with an exactly computable faithful representation $\varphi:G\to GL_n(\mathbb{Q}(X))$. Then, the (nonnegative) power orbit problem in $G$ is decidable.
\end{corollary}
\begin{proof}
    Let $M$ be the Turing machine that decides whether there exists $i\in\mathbb{N}$ with $b^i = c$ for any $b,c\in GL_n(\mathbb{Q}(X))$. Let $\varphi:G\to GL_n(\mathbb{Q}(X))$ be the exactly computable faithful representation of $G$. We describe a Turing machine $T$ as follows: Let $x, y\in G$. Compute $\varphi(x)$ and $\varphi(y)$. Compute $
    \varphi(x)^{-1}$. Run $M$ on $(\varphi(x),\varphi(y)
    )$ and on $( \varphi(x)^{-1},\varphi(y))$. If $M$ accepts, accept; if it rejects, reject.

    Suppose $T$ accepts $( x,y)$. Then, there exists an $m\in\mathbb{Z}$ such that $\varphi(x)^m = \varphi(y)$. Because $\varphi$ is an injection, this implies $x^m = y$. Now suppose there exists an $m\in\mathbb{Z}$ such that $x^m = y$. Then, $\varphi(x)^m=\varphi(x^m) =\varphi(y)$, so $T$ accepts $( x,y)$.

    Notice that the above algorithm can be modified to decide the nonnegative power orbit problem.
\end{proof}
\begin{corollary}
    The (nonnegative) power orbit problem for $B_n$ is decidable for all $n\in\mathbb{N}^+$.
\end{corollary}
\begin{proof}
    Fix $n\in\mathbb{N}^+$. We will work with Artin's presentation of $B_n$. Let $w,v\in B_n$. To prove the claim, we must show that the problem of whether there exists a $k\in\mathbb{N}$ such that $w^k = v$ is decidable. Let $\varphi:B_n\to GL_{\binom{n}{2}}(\mathbb{Z}[q^{\pm1},t^{\pm1}])$ be the Lawrence-Krammer representation. By the proof of \cite[Theorem 1.1]{LK} $\varphi$ is faithful, and is computable by \cite[Theorem 4.1]{LK}. By \thref{indcorrep}, the power orbit problem for $B_n$ is decidable.
\end{proof}
\begin{corollary}
    The (nonnegative) power orbit problem is decidable in $\operatorname{Aut}(F_2)$.
\end{corollary}
\begin{proof}

    By \cite[Lemma 4]{vgb}, $\operatorname{Aut}(F_2)$ is generated by the following automorphisms, where $F_2$ is the free group on symbols $x$ and $y$:
\begin{align*}
    \alpha_1:x\mapsto x, y\mapsto yx^{-1}\\
    \alpha_2:x\mapsto y,y\mapsto yx^{-1}y\\
    \alpha_3:x\mapsto x,y\mapsto x^{-1}y\\
    \omega:x\mapsto x^{-1},y\mapsto y.
\end{align*}
    Define $\psi:\operatorname{Aut}(F_2)\to \operatorname{GL}_{12}(\mathbb{Q}(\sqrt[12]{t_1},\sqrt[12]{t_2}))$ as in \cite[Theorem 2]{vgb}, where $\rho_i\in\operatorname{GL}_6(t_1,t_2)$ are as in \cite[Theorem 2]{vgb} for $i\in[3]$:
    \begin{align*}
        \psi(\alpha_i) = \begin{bmatrix}
            \frac{\rho_i}{\sqrt[12]{t_1^2t_2^8}}&0\\
            0&\sqrt[12]{t_1^2t_2^8}\rho_i^{-1}
            \end{bmatrix}\\
            \psi(\omega) = \begin{bmatrix}
                0&\frac{\rho_3\rho_1}{\sqrt[6]{t_1^2t_2^8}}\\
                \sqrt[6]{t_1^2t_2^8}\rho_1^{-1}\rho_3^{-1}&0
            \end{bmatrix}.
    \end{align*}
    Notice that $\psi$ is exactly computable. By \cite[Theorem 2]{vgb}, $\psi$ is faithful, so by \thref{indcorrep}, the power orbit problem for $\operatorname{Aut}(F_2)$ is decidable.
\end{proof}


\begin{corollary}
    The (nonnegative) power orbit problem is decidable in the free product $\langle v,w;v^3=w^2=1\rangle$.
\end{corollary}
\begin{proof}
    Recall that the Burau representation of $B_4$ is faithful by \cite{burau}. By \cite[Proposition 5]{dyer} and the faithfulness of the Burau representation of $B_4$, there exists an exactly computable faithful representation $\varphi:\langle v,w;v^3=w^2=1\rangle\to GL_3(\mathbb{Q}(u))$, where $u$ is an indeterminant. The result follows from \thref{indcorrep}. 
\end{proof}
\section{Outlook}\label{outlook}
There are several avenues for future work. For example, one could attempt to generalize \thref{touseLk} to $K(X)$, where $K$ is some arbitrary algebraic number field. One could also attempt to generalize \thref{touseLk} to $K(X)$ where $K$ has positive characteristic. Another project would be to determine if there exists a polynomial-time algorithm to decide the orbit problem over $\mathbb{Q}(X)$. Finally, one could classify the groups for which the power orbit problem is decidable.
\section*{Acknowledgments}
The second author would like to thank the 2023-24 Rose Research Fellows program for its guidance.

\end{document}